\documentclass[12pt]{amsart}

\input xy
\xyoption{all}

\usepackage{geometry}
\usepackage{amsmath}
\usepackage{amssymb}

\usepackage{comment}

\usepackage{amsthm}  
\usepackage{mathtools}
\usepackage{amsmath, amsthm, amssymb}
\usepackage{listings}
\usepackage{url}
\usepackage{hyperref}
\usepackage[utf8]{inputenc}
\usepackage{graphicx}
\usepackage{tikz-cd}
\usepackage{cite}
\theoremstyle{definition}
\newtheorem{mydef}{Definition}[section]
\newtheorem{lem}[mydef]{Lemma}
\newtheorem{thm}[mydef]{Theorem}

\newtheorem{question}[mydef]{Question}
\newtheorem{hypothesis}[mydef]{Hypothesis}

\newtheorem{defin}[mydef]{Definition}
\newtheorem{example}[mydef]{Example}
\newtheorem{remark}[mydef]{Remark}
\newtheorem{notation}[mydef]{Notation}
\newtheorem{fact}[mydef]{Fact}

\newcommand{\cf}{\text{cf }}

\newcommand{\ded}{\mbox{ded } }

\newcommand{\concat}{%
  \mathord{
    \mathchoice
    {\raisebox{1ex}{\scalebox{.7}{$\frown$}}}
    {\raisebox{1ex}{\scalebox{.7}{$\frown$}}}
    {\raisebox{.5ex}{\scalebox{.5}{$\frown$}}}
    {\raisebox{.5ex}{\scalebox{.5}{$\frown$}}}
  }
}

\newbox\noforkbox \newdimen\forklinewidth
\forklinewidth=0.3pt \setbox0\hbox{$\textstyle\smile$}
\setbox1\hbox to \wd0{\hfil\vrule width \forklinewidth depth-2pt
 height 10pt \hfil}
\wd1=0 cm \setbox\noforkbox\hbox{\lower 2pt\box1\lower
2pt\box0\relax}
\def\unionstick{\mathop{\copy\noforkbox}\limits}
\def\nf{\unionstick}

\setbox0\hbox{$\textstyle\smile$}
\setbox1\hbox to \wd0{\hfil{\sl /\/}\hfil} \setbox2\hbox to
\wd0{\hfil\vrule height 10pt depth -2pt width
               \forklinewidth\hfil}
\newbox\doesforkbox
\setbox\doesforkbox\hbox{\lower 0pt\box1 \lower
2pt\box2\lower2pt\box0\relax}

\def\1nf{\unionstick^{(1)}}

\def\2nf{\unionstick^{(2)}}

\def\3nf{\unionstick^{(3)}}

\def\forkindep{\mathrel{\raise0.2ex\hbox{\ooalign{\hidewidth$\vert$\hidewidth\cr\raise-0.9ex\hbox{$\smile$}}}}}

\newcommand{\indep}[4]{#2 \overset{#4}{\underset{#1}{\forkindep}}  #3}

\newcommand{\type}{\textbf{gtp}}

\title{An NIP-like Notion in abstract elementary classes}

\date{\today\\
AMS 2010 Subject Classification: Primary:  03C45, 03C48. Secondary: 03C52.} 
\keywords{Abstract Elementary Classes; forking; Classification Theory; NIP; good frames}

\author{Wentao Yang}
\email{ndwyang@imu.edu.cn}
\urladdr{http://wen-tao-y.github.io}
\address{School of Mathematical Sciences, Inner Mongolia University, Hohhot, Inner Mongolia, China}

\begin{document}

\begin{abstract}
This paper is a contribution to the ``neo-stability'' type of result for abstract elementary classes. Under certain set theoretic assumptions, we propose a definition and a characterization of NIP in AECs. The class of AECs with NIP properly contains the class of stable AECs\footnotemark. We show that for an AEC $K$ and $\lambda\geq LS(K)$, $K_\lambda$ is NIP if and only if there is a notion of nonforking on it which we call a w*-good frame. On the other hand, the negation of NIP leads to being able to encode subsets.

\end{abstract}
\maketitle

\tableofcontents
\stepcounter{footnote}
\footnotetext{See Examples \ref{ex1} and \ref{ex2} for AECs that are unstable, not elementary but NIP.}
\section{Introduction}

There is a massive body of literature on ``neostability'' for first order theories dedicated to exploration and study of forking-like relations for various classes of unstable theories. The main classes: NIP theories, simple theories, theories with the strict order property, theories with the tree property of type 1 and 2, were all presented by Shelah in \cite{shelahfobook78}. In mid 1976 Shelah set the program which he named \textbf{classification theory for non-elementary classes}. A few years later the focus shifted to abstract elementary classes (AECs).

An appropriate generalization of stability for AECs was introduced in \cite{sh394} building on many previous papers including \cite{sh12} and \cite{grsh}. In the last forty years starting with \cite{sh238} much was discovered about analogues of superstability. See \cite{v2}, \cite{gv}, and \cite{leung3} for some recent work.

In this paper we propose progress towards ``neostability of AECs'', more precisely, exploring an analogue of NIP and its negation. We propose a definition (under a certain cardinal arithmetic axiom) of NIP. Using techniques from papers by Shelah \cite{sh600}, Jarden and Shelah \cite{jrsh875} and Mazari-Armida \cite{mazariwgood}, we obtain a characterization of NIP in AECs using frames (a forking-like relation).

The notion of the $\lambda$-stable AEC was first studied in \cite{sh394} using non-splitting. Various frameworks of forking-like relations were introduced. In \cite{sh600}, Shelah introduced the local notion of the good $\lambda$-frame, an axiomatization of forking-like relations for structures of cardinality $\lambda$ in AECs, as a parallel of superstability. In \cite{BGforking} Boney and Grossberg established that for ``nice'' AECs, stablity implies existence of strong independence relations on the subclass of saturated models, which allows types of arbitrary length. In \cite{BGKV} it was shown that this relation and several others are unique/canonical (if they exist). 

Although good $\lambda$-frames are nice and powerful, sometimes they might not exist. There are several weaker notions, where some of the axioms of a good $\lambda$-frame are weakened or dropped. Vasey worked with good$^-$ $\lambda$-frames in \cite{v2} and good$^{-S}$ $\lambda$-frames in \cite{v1}.
Jarden and Shelah defined semi-good $\lambda$-frames in \cite{jrsh875}. Mazari-Armida introduced w-good $\lambda$-frames in \cite{mazariwgood}, which is weaker than all the axiomatic frames mentioned above.

\begin{mydef} 
Let $K$ be an AEC, $\lambda\geq LS(K)$. $K_\lambda$ has NIP if for all $M\in K_\lambda$, $|gS(M)|\leq  \ded \lambda .$
   
\end{mydef}

Our definition of NIP will be discussed further in the next section.

Our main results are:
\begin{thm} [$2^{\lambda^+}>2^\lambda$] Let $K$ be an AEC categorical in $\lambda\geq LS(K)$ , and $1\leq I(\lambda^+,K)<2^{\lambda^+}$. $K_\lambda$ has NIP if and only if there is a w*-good $\lambda$-frame on $K$ except possibly without (Continuity). Moreover, 
\begin{enumerate}
    \item ($\ded \lambda=\lambda^+<2^\lambda$) If $\mathfrak s_{\lambda-unq}$ is $\lambda$-compact, then the w*-good frame satisfies in addition that if $p\in S^{bs}(M)$, then there is $N\geq_K M$ and $q\in S^{bs}(N)$ extending $p$ that does not fork over $N$. In particular, for any $N'\geq_K N$ there is $q'\in gS(N')$ extending $q$ that does not fork over $N$.
    \item if $K$ is $(<\lambda^+,\lambda)$-local, then $\mathfrak s_{\lambda-unq}$ has (Continuity).
\end{enumerate}
\end{thm}

\begin{thm}
Suppose $K$ is $(<\aleph_0)$-tame, $M\in K$, $C\subseteq |M|$, $\lambda:=|C|\geq \beth_3(LS(K))$ and $(\ded \lambda)^{2^{LS(K)}}= \ded \lambda$. Suppose $|gS^1(C;M)|>\ded \lambda$. Then there is $N\in K$, $\langle \bar a_n\in^m |N|\mid n<\omega \rangle$ and $\phi$ in the language of Galois Morleyization
such that for every $w\subseteq \omega$ there is $b_w\in |N|$ such that for all $i<\omega$, 
$$
N\models \phi(\bar a_i,b_w) \iff i\in w
$$
\end{thm}

\begin{thm}
    If $K$ can encode subsets of $\mu:=\beth_{(2^{LS(K)})^+}$, then it can encode subsets of any cardinal. That is, if there are $M\in K$, $\{a_i\mid i<\mu\}\subseteq |M|$, $\{b_w\mid w\subseteq \mu\}\subseteq |M|$ such that for all $w\subseteq \mu$, $$i\in w\iff \phi(a_i,b_w),$$ then we can replace $\mu$ above by any cardinal.
\end{thm}
This paper was written while working on a Ph.D. thesis under the direction of Rami Grossberg at Carnegie Mellon University, and I would like to thank Professor Grossberg for his guidance and assistance in my research in general and in this work
specifically. I would also like to thank Jeremy Beard for checking the proofs and John Baldwin, Jeremy Beard, Will Boney, Artem Chernikov, James Cummings, Samson Leung, Marcos Mazari-Armida, Pedro Marun and Andrés Villaveces for their help, comments and suggestions. I would like to thank the referee for many insightfull comments and suggestions, and raising Questions \ref{q3}, \ref{q1} and \ref{q2}.

It is interesting to comment that Shelah already implicitly discussed similar results in \cite{sh576} dealing with Grossberg's question ``Does  $I(\lambda, K)=I(\lambda^{++},K)=1$ imply $K_{\lambda^{++}}\neq \emptyset$'' and in its followup \cite{sh600}, Chapter II of \cite{shelahaecbook}, and \cite{Sh:E46}, Chapter VI of \cite{shelahaecbook2}. 
More specifically, in \cite[VI.2.3]{shelahaecbook2} and \cite[VI.2.5]{shelahaecbook2} Shelah considered the number of branches of a tree as a bound of Galois types over a model.
\section{Preliminaries}

\begin{notation}\leavevmode
\begin{enumerate}
    \item For any structure $M$ in some language, we denote its universe by $|M|$, and its cardinality by $\|M\|$.
    \item For cardinals $\lambda$ and $\mu$, $[\lambda,\mu):=\{\kappa \in \mbox{Card}\mid \lambda\leq \kappa <\mu\}$. $[\lambda,\infty):=\{\kappa \in \mbox{Card}\mid \lambda\leq \kappa \}$.
    \item $K_{[\lambda,\mu)}:=\{M\in K\mid \|M\|\in [\lambda,\mu)\}$. $K_\lambda:=K_{[\lambda,\lambda^+)}$
\end{enumerate}

\end{notation}

\begin{mydef}
For $K$ an AEC, we say:
\begin{enumerate}
    \item $K$ has the amalgamation property (AP) if for all $M_0\leq M_\ell$ for $\ell=1,2$, there is $N\in K$ and $K$-embeddings $f_\ell:M_\ell\to N$ for $\ell=1,2$ such that $f_1 \restriction_{M_0}=f_2 \restriction_{M_0}$.
    \item $K$ has the joint embedding property (JEP) if for all $M_0$, $M_1\in K$ there are $N\in K$ and $K$-embeddings $f_\ell:M_l\to N$ for $\ell=0,1$. 
    \item $K$ has no maximal models (NMM) if for all $M\in K$ there is $N>_K M$.
\end{enumerate}
\end{mydef}

\begin{remark}
  For a property $P$, e.g. amalgamation, we say that $K_\lambda$ has $P$ or that $K$ has $\lambda$-P if we restrict to $K_\lambda$ in the above definition. 
\end{remark}

\begin{mydef}\leavevmode 
\begin{enumerate}
    \item $K^3_\lambda:=\{(a,M,N)\mid M, N\in K_\lambda, M<_K N, a\in |N|-|M|\}$.
    \item For $(a_0,M_0,N_0)$, $(a_1,M_1,N_1)\in K^3_\lambda$, we say $(a_0,M_0,N_0)\leq (a_1, M_1,N_1)$ if $M_0\leq M_1$, $a_0=a_1$ and $N_0\leq_K N_1$.
    \item For $(a_0,M_0,N_0)$, $(a_1,M_1,N_1)\in K^3_\lambda$ and $K$-embedding $h:N_0\to N_1$, $(a_0,M_0,N_0)\leq_h (a_1, M_1, N_1)$ if $h\restriction_{M_0}:M_0\to M_1$ is a $K$-embedding and $h(a_0)=a_1$.
    
\end{enumerate}
\end{mydef}

\begin{mydef}\leavevmode\label{galoistypes}
\begin{enumerate}
    \item For $(a_0,M_0,N_0), (a_1,M_1,N_1)\in K^3_\lambda$, $(a_0,M_0,N_0) E_{at}(a_1,M_1,N_1)$ if $M_0=M_1$, and there are $N\in K$, $f_0:N_0\to N$, and $f_1:N_1\to N$ $K$-embeddings such that $f_0(a_0)=f_1(a_1)$ and $f_0\restriction_{M_0}=f_1\restriction_{M_0}$.
    \item $E$ is the transitive closure of $E_{at}$.\
    \item For $(a,M,N)\in K^3_\lambda$, \emph{the Galois type of $a$ over $M$ in $N$} is $\type(a/M,N):=[(a,M,N)]_{E}$.
    \item For $M\in K_\lambda$, $gS(M):=\{\type(a/M,N)\mid (a,M,N)\in K^3_\lambda\}$.
    \end{enumerate}
    \item For $M_0\leq_K M\in K_\lambda$ and $p=\type(a/M,N)\in gS(M)$, define $p\restriction_{M_0}:=\type(a/M_0,N)$.
    \item For $M_0\leq_K M_1$ and types $p\in gS(M_0)$ and $q\in gS(M_1)$, we say $p\leq q$ if $p=q\restriction_{M_0}$.
\end{mydef}

 \begin{remark}
   If $K_\lambda$ has AP then $E_{at}=E$.
 \end{remark}

\begin{mydef}
Assume that $K_\lambda$ has AP. For $M$, $N\in K$, $p\in gS(M)$ and $K$-embedding $h:M\to N$, we define $h(p):=\type(h'(a)/h[M], N)$, where $h':M'\to N'$ extends $h$ and $(a,M,M')\in p$. Note that $h(p)$ does not depend on the choice of $(a,M,M')$ or $h'$. See \cite[3.1]{leung3} for a proof.
\end{mydef}

\begin{mydef}
Let $\langle M_i\mid i<\delta\rangle$ be increasing continuous. A sequence of types $\langle p_i\in gS(M_i)\mid i<\delta\rangle$ is coherent if there are $(a_i,N_i)$ for $i<\delta$ and $f_{j,i}:N_j\to N_i$ for $j<i<\delta$ such that:
\begin{enumerate}
    \item $f_{k,i}=f_{j,i}\circ f_{k,j}$ for all $k<j<i$.
    \item $\type(a_i/M_i,N_i)=p_i$.
    \item $f_{j,i}\restriction_{M_j}=id_{M_j}$.
    \item $f_{j,i}(a_j)=a_i$.
\end{enumerate}
\end{mydef}

The notion of coherent sequence of types first appeared in \cite[2.12]{GVcanonical}, Here we use the version in \cite[3.14]{mazariwgood} that avoids the use of a monster model.

 \begin{fact}\cite[12.3]{baldwinbook}
Let $\delta$ be a limit ordinal and $\langle M_i\in K\mid i\leq\delta\rangle$ increasing continuous, and $\langle p_i\in gS(M_i)\mid i<\delta\rangle$ a coherent sequence of types. Then there is $p\in gS(M_\delta)$ an upper bound of $\langle p_i\in gS(M_i)\mid i<\delta\rangle$, where the order is the one from Definition \ref{galoistypes}(5). 
 
 \end{fact}

\begin{fact}\cite[11.3(2)]{baldwinbook}
Let $\delta$ be a limit ordinal, $\langle M_i\in K\mid i\leq\delta\rangle$ increasing continuous, and $\langle p_i\in gS(M_i)\mid i<\delta\rangle$ a sequence of types with upper bound $p\in gS(M_\delta)$. Then there are $\langle N_i\mid i\leq \delta\rangle$ and $\langle f_{j,i}\mid j<i\rangle$ that witness $\langle p_i\in gS(M_i)\mid i\leq\delta\rangle$ being a coherent sequence.
\end{fact}

\begin{mydef}
 \cite[0.22(2)]{sh576} Let $\mu>\lambda$. $N\in K_{\mu}$ is \emph{saturated in $\mu$ above $\lambda$} if for all $M\leq_K N$, $\lambda\leq \|M\|<\mu$, $N$ realizes $gS(M)$.
\end{mydef}

\begin{mydef}\cite[0.26(1)]{sh576}
Let $\mu>\lambda$. $N\in K_{\mu}$ is \emph{homogeneous in $\mu$ for $\lambda$} if for all $M_1\leq_K N$, $M_1\leq_K M_2\in K_\lambda$, $\lambda\leq \|M_1\|\leq \|M_2\|<\mu$, there is $K$-embedding $f:M_2\to N$ above $M_1.$
\end{mydef}

\begin{fact}\label{satmh}\cite[0.26(1)]{sh576} Let $\mu>\lambda$. If $K_\lambda$ has AP then $M\in K_{\mu}$ is saturated over $\mu$ for $\lambda$ if and only if $M$ is homogeneous over $\mu$ for $\lambda$.
  
\end{fact}

\begin{mydef}\cite{sh10}
For a cardinal $\lambda$, 
\begin{align*}
    \ded \lambda :=\sup\{& \kappa \mid \exists \mbox{ a regular } \mu \mbox{ and a tree $T$ with $\leq \lambda$ nodes and $\kappa$ branches of }\\ & \mbox{length } \mu\}.
\end{align*}

\end{mydef}

\begin{fact}\cite[II.4.11]{shelahfobook78}\label{fonip}
Let $T$ be a complete first order theory and $\phi$ a formula in its language. $\lambda$ is an infinite cardinal such that $2^\lambda>\ded \lambda$. The following are equivalent:
\begin{enumerate}
    \item $\phi$ has the independence property.
    \item $|S_\phi(A)|> \ded |A|$ for some infinite set $A$, $|A|=\lambda$.
    \item $|S_\phi(A)|=2^{|A|}$ for some infinite set $A$, $|A|=\lambda$.
\end{enumerate}
\end{fact}
\begin{fact}\cite[II.4.12]{shelahfobook78}
    Let $T$ be a complete theory in countable language, and $f_T(\lambda):=\sup\{|S(M)|\mid M\models T$, $\|M \|=\lambda\}$. Then $f_T(\lambda)$ is exactly one of: $\lambda$, $\lambda+2^{\aleph_0}$, $\lambda^{\aleph_0}$, $\ded \lambda$, $(\ded \lambda)^{\aleph_0}$ or $2^\lambda$. See also \cite{keisler}.
\end{fact}
It is reasonable to propose the following definition:
\begin{mydef} 
Let $K$ be an AEC, $\lambda\geq LS(K)$. $K_\lambda$ has NIP if for all $M\in K_\lambda$, $|gS(M)|\leq \ded \lambda .$
   
\end{mydef}

At present it is unclear that we have discovered the ``correct'' notion. In fact, it is plausible that there are several different notions that are equivalent when $K$ is an elementary class, but distinct for some non-elementary $K$. One weakness of our definition is that unlike the corresponding first order notion, it is probably not absolute.

Grossberg raised the following question:

\begin{question}
Is there an equivalent notion which does not rely on extra set theoretic assumptions. (at least for AECs $K$ with $LS(K)=\aleph_0$ which are also $PC_{\aleph_0}$)?
\end{question}

\begin{question} \label{q3}
    Is there a global characterization of NIP?
\end{question}

\begin{fact}\cite[2.5.8]{jrsh875}\label{jardenshelah} Assume $K$ has JEP, AP and NMM. Suppose there is $S^{bs}\subseteq gS$ family of types on $K$ satisfying only (Density), (Invariance), and for all $M\in K_\lambda$, $|S^{bs}(M)|\leq \lambda^+$. See Definitions \ref{preframe} and \ref{w*-good}.

\begin{enumerate}
    \item If $\langle M_\alpha\in K_\lambda\mid \alpha<\lambda^+ \rangle$ is increasing and continuous, and there is a stationary set $S\subseteq \lambda^+$ such that for every $\alpha\in S$ and every model $N$, with $M_\alpha \leq_K N$, there is a type $p\in S^{bs}(M_\alpha)$ which is realized in $M_{\lambda^+}:=\bigcup_{i<\lambda^+} M_i$ and in $N$, then $M_{\lambda^+}$ is saturated in $\lambda^+$ above $\lambda$.
    \item For all $M\in K_\lambda$, $|gS(M)|\leq \lambda^+$.
\end{enumerate}

\end{fact}

The following is an example of an AEC satisfying NIP that is not elementary or stable.

\begin{example}\cite[2.2.4]{jrsh875}\label{ex1}
Let $\lambda$ be a cardinal. Let $P$ be a family of $\lambda^+$ subsets of $\lambda$. Let $\tau := \{R_\alpha : \alpha < \lambda\} $ where each $R_\alpha$ is an unary predicate. Let $K$ be the class of models $M$ for $\tau$ such that for each $a\in |M|$, $\{ \alpha\in \lambda \mid M \models R_\alpha(a)\} \in P$. Note that $K$ is not elementary. Let $\leq_K$ be the substructure relation on $K$. The trivial $\lambda$-frame on $K_\lambda$ satisfies the axioms of a semi-good $\lambda$-frame\cite[2.1.3]{jrsh875}, so in particular by Fact \ref{jardenshelah} $K_\lambda$ satisfies NIP. On the other hand, it is unstable.
\end{example}

The next is an algebraic example of an AEC that satisfies NIP and is not elementary or stable.

\begin{example}\label{ex2} ($\ded \lambda=(\ded \lambda)^{\aleph_0}$)
Let $K$ be the class of divisible ordered abelian groups (denoted by $K'$) omitting the type of an infinitesimal element. That is, a model $G$ of this class $K$ is an Archimedean divisible ordered group (for each $x\in G$ and $n\in \mathbb N$, there is $y\in G$ such that $ny=x$). The class $K'$ (before omitting the type) admits quantifier elimination.  The order makes $K'$ unstable. Because $K'$ is NIP (in the sense of first order model theory), its number of syntactic types is bounded by $\ded \lambda$. Because of QE, the Galois types (viewing $K'$ as an AEC) agree with first order syntactic types. For every pretype in $K'$, denoted by $(a,M,N)$, such that $a$ is not infinitesimal and $M$ is archimedean, we can find $N'$ such that $a\in |N|$, $|M|\subseteq |N|$ by taking the divisible hull of $a\cup |M|$. Thus the Galois types in $K$ are no more than those in $K'$. Thus $K$ is an AEC that is NIP, unstable and non-elementary.
\end{example}

\begin{question} \label{q1}
    Are there other (perhaps more interesting than the previous two) examples of NIP that arise naturally in algebra or analysis?
\end{question}

\begin{mydef}\cite[VI.1.12(1)]{shelahaecbook2} We say $S_*$ is a \emph{$\leq_{K_\lambda}$-type-kind} when:
\begin{enumerate}
    \item $S_*$ is a function with domain $K_\lambda$.
    \item $S_*(M)\subseteq gS(M)$ for all $M\in K_\lambda$.
    \item $S_*(M)$ commutes with isomorphisms.
\end{enumerate}
\end{mydef}

\begin{mydef}\cite[VI.2.9]{shelahaecbook2}
\begin{enumerate}
    \item For $M\in K$ and $\Gamma\subseteq gS(M)$, $\Gamma$ is \emph{inevitable} if for all $N>_K M$ there is $a\in |N|-|M|$ with $\type(a/M,N)\in \Gamma$.
    \item For $M\in K$ and $\Gamma\subseteq gS(M)$, $\Gamma$ is \emph{$S_*$-inevitable} if for all $N>_K M$, if there is $p\in S_*(M)$ realized in $N$ then there is $q\in \Gamma$ realized in $N$.
\end{enumerate}

\end{mydef}

\begin{mydef}\cite[VI.1.12(2)]{shelahaecbook2} For $\leq_{K_\lambda}$-type-kinds $S_1$ and $S_2$, say $S_1$ is \emph{hereditarily in $S_2$} when: for $M\leq_K N$ and $p\in S_2(N)$ we have $p\restriction_M\in S_1(M)\implies p\in S_1(N)$. When $S_2$ is just $gS$ (all types) we omit ``in $S_2$'' and say $S_1$ is hereditary.
\end{mydef}

\begin{mydef}
Let $M\in K_\lambda$. $p\in gS(M)$ is \emph{$<\mu$-minimal} if for all $M\leq N\in K_\lambda$, $|\{q\in gS(N): q\restriction_M=p\}|<\mu$. $$S^{<\mu-min}(M):=\{p\in gS(M)\mid p\mbox{ is }<\mu \mbox{-minimal}\}.$$
\end{mydef}
\begin{remark}
  $S^{<\mu-min}$ and $S^{\lambda-al}$ (defined in Lemma \ref{allemma}) are hereditarily in $gS$.
\end{remark}

The following principle known as the weak diamond was introduced by Devlin and Shelah \cite{sh65}. 
\begin{defin}
Let $S\subseteq \lambda^+$ be a stationary set.  $\Phi_{\lambda^+}^{2}(S)$ holds if and only if for all $ F: (2^{\lambda}) ^{<\lambda^+} \rightarrow 2$  there exists $ g: \lambda^+ \rightarrow 2$ such that for all $ f: \lambda^+ \rightarrow 2^{\lambda}$ the set $\{ \alpha \in S : F(f \upharpoonright_{\alpha}) = g(\alpha) \}$ is stationary. 
\end{defin}

\begin{fact}\label{nonstructure2}\cite[VI.2.18]{shelahaecbook2}
($2^{\lambda} < 2^{\lambda^+}$)
Assume $K$ has amalgamation and no maximal model in $\lambda$.    If
\begin{enumerate}
\item $S_*$ is $\leq_{K_\lambda}$-type-kind and hereditary,
    \item $S_*\subseteq gS^{<\lambda^+-min}$,  and
    \item There is $M \in K_\lambda$ such that:
    \begin{enumerate}
    \item $|gS_*(M)|\geq \lambda^+$, and
    \item if $M \leq_K N\in K_\lambda$, no subset of $S_*(N)$ of size $\leq \lambda$ is $S_*$-inevitable,
\end{enumerate}     

\end{enumerate}
then $I(\lambda^+,K)=2^{\lambda^+}$.
\end{fact}

\begin{fact}\label{nonstructure3}\cite[VI.2.11(2)]{shelahaecbook2}\footnote{A complete argument of this result does not appear in \cite{shelahaecbook2}. A sketch of the argument can be found in a forthcoming paper with Marcos Mazari-Armida using Sebastien Vasey's argument in \cite{vasey_invitation}}
For every $M\in K_\lambda$ we have $|S_*(M)|\leq \lambda$ when:
\begin{enumerate}
    \item $K$ has AP in $\lambda$, and
    \item $S_*$ is a hereditary $\leq_{K_\lambda}$-type-kind in $gS$, and
    \item For every $M\in K_\lambda$ there is an $S_*$-inevitable $\Gamma_M\subseteq gS(M)$ of cardinality $\leq \lambda$.
\end{enumerate}
\end{fact}

\section{The w*-good frame}
In this section we define w*-good frames, and show that $K_\lambda$ has NIP if and only if $K$ has a w*-good $\lambda$-frame under additional assumptions. We work with an AEC $K$ and $\lambda\geq LS(K)$.

\begin{mydef}\label{preframe} \cite[III.0]{shelahaecbook}
Let $\lambda<\mu$, where $\lambda$ is a cardinal, and $\mu$ is a cardinal or $\infty$. A \emph{pre-$[\lambda,\mu)$-frame} is a triple $\mathfrak s=(K,\nf{},S^{bs})$ such that:
\begin{enumerate}
    \item $K$ is an AEC with $\lambda\geq LS(K)$ and $K_\lambda\neq \emptyset$.
    \item $S^{bs}\subseteq \bigcup_{M\in K_{[\lambda,\mu)}}gS(M)$. Let $S^{bs}(M):=gS(M)\cap S^{bs}$. Types in this family are called \emph{basic types}.
    \item $\nf$ is a relation on quadruples $(M_0,M_1,a,N)$, where $M_0\leq_K M_1\leq N$, $a\in |N|$ and $M_0,M_1,N\in K_{[\lambda,\mu)}$. We write $\indep{M_0}{a}{M_1}{N}$, or we say $\type(a/M_1,N)$ does not fork over $M_0$ when the relation $\nf$ holds for $(M_0,M_1,a,N)$.
    \item (Invariance) If $f:N\cong N'$ and $\indep{M_0}{a}{M_1}{N}$, then $\indep{f[M_0]}{f(a)}{f[M_1]}{N'}$. If $\type(a/M_1,N)\in S^{bs}(M_1)$, then $\type(f(a)/f[M_1],N')\in S^{bs}(f[M_1])$.
    \item (Monotonicity) If $\indep{M_0}{a}{M_1}{N}$ and $M_0\leq_K M_0'\leq_K M_1'\leq_K M_1\leq_K N'\leq_K N\leq_K N''$ with $N''\in K_{[\lambda,\mu)}$ and $a\in |N'|$, then $\indep{M_0'}{a}{M_1'}{N'}$ and $\indep{M_0'}{a}{M_1'}{N''}$.
    \item (Non-forking Types are Basic) If $\indep{M}{a}{M}{N}$ then $\type(a/M,N)\in S^{bs}(M).$
\end{enumerate}
\end{mydef}

\begin{mydef}\cite[3.6]{mazariwgood}\label{w-good}
A pre-$[\lambda,\mu)$-frame $\mathfrak s= (K,\nf, S^{bs})$ is a \emph{w-good frame} if:
\begin{enumerate}
    \item $K_{[\lambda,\mu)}$ has AP, JEP and NMM.
    \item (Weak Density) For all $M<_K N\in K_\lambda$, there is $a\in |N|-|M|$ and $M'\leq N'\in K_{[\lambda,\mu)}$ such that $(a,M,N)\leq (a,M',N')$ and $\type(a/M',N')\in S^{bs}(M')$.
    \item (Existence of Non-Forking Extension) If $p\in S^{bs}(M)$ and $M\leq_K N$, then there is $q\in S^{bs}(N)$ extending $p$ which does not fork over $M$.
    \item (Uniqueness) If $M\leq_K N$ both in $K_{[\lambda,\mu)}$, $p,q\in S^{bs}(N)$ both do not fork over $M$, and $p\restriction_M=q\restriction_M$, then $p=q$.
    \item (Strong Continuity\footnote{This was called just continuity in \cite{mazariwgood}. The author would like to thank Marcos Mazari-Armida for pointing out that the continuity axiom of a good frame requires only the moreover part.}) If $\delta<\mu$ a limit ordinal, $\langle M_i\mid i\leq \delta\rangle$ increasing and continuous, $\langle p_i\in S^{bs}(M_i)\mid i<\delta \rangle$, and $i<j<\delta$ implies $p_j\restriction M_i=p_i$, and $p_\delta\in S(M_\delta)$ is an upper bound for $\langle p_i\mid i<\delta\rangle$, then $p_\delta\in S^{bs}(M_\delta)$. Moreover, if each $p_i$ does not fork over $M_0$ then neither does $p_\delta$.
\end{enumerate}
\end{mydef}

\begin{mydef}\label{w*-good}
A pre-$[\lambda,\mu)$-frame $\mathfrak s=(K,\nf,S^{bs})$ is a \emph{w*-good} frame if $\mathfrak s$ satisfies:
\begin{enumerate}
    \item $K_{[\lambda,\mu)}$ has AP, JEP and NMM.

    \item (Uniqueness). See Definition \ref{w-good}.
    \item (Basic NIP) For all $M\in K_{[\lambda,\mu)}$ $|S^{bs}(M)|\leq \ded \|M\|$.
    \item (Few Non-Basic Types) For all $M\in K_{[\lambda,\mu)}$, $|gS(M)-S^{bs}(M)|\leq \lambda.$
    \item (Continuity\footnote{This is the continuity axiom for good frames.}) Let $\delta<\mu$ be a limit ordinal, $\langle M_i\mid i\leq \delta\rangle$ increasing and continuous, $\langle p_i\in S^{bs}(M_i)\mid i<\delta \rangle$, and $i<j<\delta$ implies $p_j\restriction_{M_i}=p_i$, and $p_\delta\in gS(M_\delta)$ is an upper bound for $\langle p_i\mid i<\delta\rangle$. If each $p_i$ does not fork over $M_0$ then $p_\delta\in S^{bs}(M_\delta)$ and $p_\delta$ also does not fork over $M_0.$
    \item (Transitivity) if $p\in S^{bs}(M_2)$ does not fork over $M_1\leq_K M_2$, and $p\restriction_{M_1}$ does not fork over $M_0\leq_K M_1$, then $p$ does not fork over $M_0$.
    
\end{enumerate}
\end{mydef}

Although the author cannot find a proof or counterexample, w-good and w*-good frames are likely to be incomparable.

\begin{remark}
(Continuity) is weaker than (Strong Continuity). Without not forking over $M_0$ one cannot deduce that $p_\delta\in S^{bs}(M_\delta).$
\end{remark}

\begin{remark}
  In a w-good frame (Transitivity) is implied by several other properties including (Existence of Non-Forking Extension). For a w*-good frame, where (Existence of Non-Forking Extension) does not hold in general, we need to explicitly include (Transitivity) as an axiom.
\end{remark}

\begin{mydef}
When $\mu=\lambda^+$ in the previous definitions, we say $\mathfrak s$ is a pre-/w-good/w*-good $\lambda$-frame.
\end{mydef}

From now on we build a w*-good $\lambda$-frame on $K$ assuming the following:

\begin{hypothesis}[$2^{\lambda^+}>2^\lambda$]
We fix $K$ an AEC and a cardinal $\lambda\geq LS(K)$ such that $K$ is categorical in $\lambda$. Further more $1\leq I(\lambda^+,K)<2^{\lambda^+}$, and $K_\lambda$ has NIP.
\end{hypothesis}

As $K$ is categorical in $\lambda$, then $K$ has $\lambda$-AP by the following fact, which appeared in \cite[3.5]{sh88} first, and a clearer proof can be found in \cite[4.3]{Gclassification}. $\lambda$-JEP follows from categoricity, and $\lambda$-NMM follows from categoricity and $K_{\lambda^+}\neq \emptyset.$

\begin{fact}\cite[3.5]{sh88}
($2^\lambda<2^{\lambda^+}$) If $I(\lambda, K)=1\leq I(\lambda^+,K)<2^{\lambda^+}$, then $K$ has the $\lambda$-AP.
\end{fact}

\begin{mydef}
$p=\type(a/M,N)$ has the extension property if for every $K$-embedding $f:M\to M_1\in K_\lambda$ there is $q\in gS(M_1)$ extending $f(p)$.
\end{mydef}

\begin{mydef}
$p=\type(a/M,N)$ is \emph{$\lambda$-unique}\footnote{This notion was first introduced by Shelah in \cite[6.1]{Sh:48}, called minimal types there. Note that this is a different notion from the minimal types of \cite{sh576}. These types are also called \emph{quasiminimal types} in the literature, see for example \cite{zilber} and \cite{les}}.  if
\begin{enumerate}
    \item $p=\type(a/M,N)$ has the extension property, and 
    \item for every $M\leq_K M' \in K_\lambda$, $p$ has at most one extension $q\in gS(M')$ with the extension property.
\end{enumerate}
\end{mydef}

\begin{fact}\cite[VI.2.5(2B)]{shelahaecbook2}\label{extension}
If $K_\lambda$ has AP and $\lambda\geq LS(K)$, $\type(a,M,N)$ has $\geq \lambda^+$ realizations in some extension of $M$ (necessarily in $K_{\geq\lambda^+})$ if and only if $ \type (a/M,N)$ has the extension property. 
\end{fact}

Now we define the w*-good $\lambda$-frame.
\begin{mydef} The preframe $\mathfrak s_{\lambda-unq}$ is defined such that:
\begin{enumerate}
    \item $S^{bs}(M):=\{p=\type(a/M,N)\mid p \mbox{ has the extension property}\}$. \item $p=\type(a/M,N)\in S^{bs}(M)$ does not fork over $M_0\leq_K M$ if $p\restriction_{M_0}$ is $\lambda$-unique. 
\end{enumerate}

\end{mydef}

\begin{lem}\label{allemma}
$S^{\lambda-al}(M):=\{p\in gS(M)\mid p\mbox{ has }\leq \lambda\mbox{-many realizations}\}$ satisfies $|S^{\lambda-al}(M)|\leq \lambda$. By realizations we mean realizations in any $\leq_K$-extension of $M$ in $K_{\lambda^+}$. So $\mathfrak s_{\lambda-unq}$ satisfies (Few Non-Basic Types).
\end{lem}
\begin{proof}
Suppose not, i.e. $|S^{\lambda-al}(M)|\geq \lambda^+$.\\
\textbf{Claim: } There is no $\Gamma\subseteq S^{\lambda-al}(M)$, $|\Gamma|\leq \lambda$ that is inevitable. 

Otherwise, suppose there exists such $\Gamma$. By Fact \ref{nonstructure3}, taking $S_*$ to be $gS$, and $\Gamma_M$ to be $\Gamma$, we have $|gS(M)|\leq \lambda$, so in particular $|S^{\lambda-al}(M)|\leq \lambda$, contradiction.

Now by the claim and Fact \ref{nonstructure2}, taking $S_*$ there to be $S^{\lambda-al}$ and $\mu$ there to be $\lambda^+$, we have $I(\lambda^+,K)=2^{\lambda^+}$, contradiction. 
\end{proof}

Thus from now on in this section we also assume $|S^{\lambda-al}(M)|\leq \lambda$.

\begin{lem}
$\mathfrak s_{\lambda-unq}$ satisfies the following properties in Definitions \ref{preframe}, \ref{w-good} and \ref{w*-good}:
\begin{enumerate}
    \item (Invariance).
    \item (Monotonicity).
    \item (Non-Forking Types are Basic).
   
    \item AP, JEP and NMM.
     \item (Basic NIP).
    \item (Uniqueness).
    \item (Transitivity).
\end{enumerate}
\end{lem}
\begin{proof}
Easy. We prove (Transitivity) as an example. Suppose $p\in S^{bs}(N)$ does not fork over $M_1\leq_K N$, and $p\restriction_{M_1}$ does not fork over $M_0\leq_K M_1$. Then $(p\restriction_{M_1})\restriction_{M_0}$ is $\lambda$-unique, i.e. $p\restriction_{M_0}$ is. Thus $p$ does not fork over $M_0$.
\end{proof}

The following property is essential for the next lemma.

\begin{mydef}
    A type family $S_*$ is $\lambda$-compact if for every limit ordinal $\delta<\lambda^+$, for every $\langle M_i\in K_\lambda : i<\delta\rangle$ an increasing continuous chain and for every  coherent sequence of types $\langle p_i\in S_*(M_i): i<\delta\rangle$, there is an upper bound $p_\delta\in S_*(\bigcup_{i < \delta} M_i)$ to the sequence such that $\langle p_i\in S_*(M_i): i<\delta +1 \rangle$ is a coherent sequence.
\end{mydef}

For certain results in this paper we need to assume that the basic types (i.e. those with the extension property) is $\lambda$-compact, which, for example, holds for AECs with the disjoint amalgamation property\footnote{For any non-algebraic type $\type(a/M,N)$, if we would like to extend it to $M\leq_K N'$, just disjointly amalgamate $N'$ and $N$ over $M$. The type of the image of $a$ is a non-algebraic extension over $N'$. Thus every type has the $\lambda$-extension property}, where every type has the extension property. Many classes of modules have the disjoint amalgamation property. See \cite[2.10]{dap} and \cite[2.2]{dapbaldwin}. Also, this assumption also holds in quasiminimal abstract elementary classes \cite[4.1]{vaseyquasiminimal}, where there is at most one non-algebraic type\footnote{Since the class is unbounded, there is a non-algebraic type over any countable model (see \cite[4.1]{vaseyquasiminimal}), so any type must have a non-algebraic to any extension of its domain (the unique non-algebraic type there).}.

\begin{lem}[$\ded \lambda=\lambda^+<2^\lambda$]\label{malem}
 Suppose that $S^{bs}$ is $\lambda$-compact. If $p\in S^{bs}(M)$, then there is $N\geq_K M$ and $q\in S^{bs}(N)$ extending $p$ that does not fork over $N$. In particular, for any $N'\geq_K N$ there is unique $q'\in gS(N')$ extending $q$ that does not fork over $N$.
\end{lem}
\begin{proof}
It suffices to show that there is a $\lambda$-unique type above any basic type. By Fact \ref{jardenshelah} let $\mathfrak C\in K_{\lambda^+}$ be saturated in $\lambda^+$ over $\lambda$. It is also homogeneous in $\lambda^+$ over $\lambda$ by Fact \ref{satmh}. Let $(a,M,N)\in K^3_\lambda$ such that $\type(a/M,N)$ has the extension property and there is no $\lambda$-unique type above $\type(a/M,N)$. Build $(a_\eta,M_\eta,N_\eta)\in K^3_\lambda$ for $\eta\in {}^{<\lambda}2$ and $h_{\eta,\nu}$ for $\eta<\nu\in {}^{<\lambda}2$ such that:
\begin{enumerate}
    \item $(a_{\langle\rangle},M_{\langle\rangle},N_{\langle\rangle})=(a,M,N)$.
    \item $(a_\eta,M_\eta,N_\eta)\leq_{h_{\eta,\nu}} (a_\nu,M_\nu,N_\nu)$ for $\eta < \nu$.
    \item $h_{\eta,\rho}=h_{\nu,\rho}\circ h_{\eta,\nu}$ for $\eta<\nu<\rho$.
    \item $M_{\eta\concat 0}=M_{\eta\concat 1}$, $N_{\eta\concat 0}=N_{\eta\concat 1}$, and $h_{\eta,\eta\concat 0}\restriction M_\eta=h_{\eta,\eta\concat 1}\restriction M_\eta$.
    \item $\type(a_{\eta\concat 0},M_{\eta\concat 0},N_{\eta\concat 0})\neq \type(a_{\eta\concat 1},M_{\eta\concat 1},N_{\eta\concat 1})$, both having $\lambda^+$-many realizations.
    \item If $\eta\in {}^\delta 2$ for $\delta$ a limit ordinal, take $M_\eta$ and $N_\eta$ to be directed colimits.
\end{enumerate}
\textbf{Construction:} Base case and limit case are clear. At successor stage use non-$\lambda$-uniqueness to get two distinct extensions, each having $\lambda^+$-many realizations. \\
\textbf{Enough:} Let $M\leq_K \mathfrak C\in K_{\lambda^+}$ be saturated over $\lambda$. Build $g_\eta:M_\eta\to \mathfrak C$ for $\eta\in {}^{\leq \lambda} 2$ such that:
\begin{enumerate}
    \item $g_\eta\circ h_{\eta,\nu}=g_\nu$ for $\nu < \eta$.
    \item $g_{\eta\concat 0}=g_{\eta\concat 1}$
\end{enumerate}
\textbf{Sub-Construction:} We carry out the construction by induction on the $\ell(\eta)$, the length of $\eta$. For the base case take $g_{\langle\rangle}$ to be inclusion $M\leq_K \mathfrak C$. At limit use the universal property of $M_\eta$ as a directed colimit. For the successor case, for $\eta$ of length $\alpha=\beta+1$, suppose we have $g_\eta$. 
\begin{equation}
\begin{tikzcd}
\mathfrak C& M''_{\eta\concat 0} \arrow[l,"j"] &M'_{\eta\concat 0} \arrow[r, "\cong_g"] \arrow[l,"\cong_h"] & M_{\eta\concat 0} \\
& g_\eta[M_\eta] \arrow[ul,"id"] \arrow[u,"id"] & M_\eta \arrow[u,"id"] \arrow[l, "\cong_{g_\eta}"]  \arrow[r,"\cong{h_{\eta, \eta\concat 0}}"]  & h_{\eta, \eta\concat 0}[M_\eta] \arrow[u,"id"]
\end{tikzcd}
\end{equation}
Use basic extension to obtain the right square and $g$, and then obtain the middle square and $h$. Finally the left triangle is by saturation of $\mathfrak C$. Now define $g_{\eta\concat 0}=g_{\eta\concat 1}$ to be the composition of the top row from right to left.

\textbf{Sub-Construction is enough: }For each branch $\eta\in {}^{\lambda}2$, take directed colimit to obtain $(a_\eta,M_\eta,N_\eta)$. Obtain $f_\eta:M_\eta\to \mathfrak C$ by the universal property of colimits such that $f_\eta\circ h_{\nu,\eta}=g_\nu$ for all $\nu<\eta$, and obtain $f'_\eta:N_\eta\to\mathfrak C$ extending $f_\eta$ by saturation over $\lambda$. Since each $f'_\eta(a_\eta)\in |\mathfrak C|$, but $\|\mathfrak C\|=\ded \lambda <2^\lambda$, there must be $\eta\neq\nu\in {}^\lambda 2$ such that $f'_\eta(a_\eta)=f'_\nu(a_\nu)$. Let $\alpha<\lambda$ be the least such that $\eta(\alpha)\neq \nu(\alpha)$. Without loss of generality say $\eta(\alpha)=0$ and $\nu(\alpha)=1$. Then the following diagram commutes:
\begin{equation}
\begin{tikzcd}
 N_{\eta\restriction_\alpha \concat 0} \ar[r, "f'_\eta\circ h_{\eta\restriction_\alpha\concat 0,\eta}"]& \mathfrak C\\
 M_{\eta\restriction_\alpha \concat 0}\ar[u,"id"] \ar[r,"id"] &  N_{\eta\restriction_\alpha \concat 1} \ar[u,"f'_\nu\circ h_{\eta\restriction_\alpha\concat 1,\nu}"]
\end{tikzcd}
\end{equation}
with $f'_\eta\circ h_{\eta\restriction_\alpha\concat 0,\eta}(a_{\eta\restriction_\alpha\concat 0})=f'_\nu\circ h_{\eta\restriction_\alpha\concat 1,\nu}(a_{\eta\restriction_\alpha\concat 1})$ since $f'_\eta(a_\eta)=f'_\nu(a_\nu)$, contradicting requirement (5) of the construction.
\end{proof}

\begin{remark}
  The proof of Lemma \ref{malem} is along the argument of Mazari-Armida in \cite[4.13]{mazariwgood} and \cite[VI.2.25]{shelahaecbook2}, and the difference is that there the saturated model over $\lambda$ lies in $K_{\lambda^{++}}$. For completeness we included all the details. 
\end{remark}
\begin{question}
    Lemma \ref{malem} is a weaker form of (Existence of Non-Forking Extension). Is it possible to obtain (Existence of Non-Forking Extension) in its full strength, by perhaps considering another family of basic types and non-forking relation? One could imitate the w-good $\lambda$-frame in \cite{mazariwgood} and use $\lambda$-unique types as basic ones, and then Lemma \ref{malem} gives a proof of (Weak Density). However, then it is hard to show that having such a frame implies NIP.
\end{question}

The following definition is \cite[1.8]{sh394}, which is defined for types of any finite length. Here we only need it for length $1$. Thus we use the version from \cite[11.4(1)]{baldwinbook}.
\begin{mydef}
\begin{enumerate}
    \item $K$ is \emph{$(\kappa,\lambda)$-local} if for every increasing continuous chain $M=\bigcup_{i<\kappa}M_i$ with $\|M\|=\lambda$ and for any $p,q\in gS(M)$: if $p\restriction_{M_i}=q\restriction_{M_i}$ for all $i$ then $p=q$.
    \item $K$ is \emph{$(< \kappa, \lambda)$-local} if $K$ is $(\mu,\lambda)$-local for all $\mu<\kappa$.
\end{enumerate}

\end{mydef}

\begin{lem} \label{locallem}If $K$ is $(<\lambda^+,\lambda)$-local, then $\mathfrak s_{\lambda-unq}$ has (Continuity).
\end{lem}
\begin{proof}
Let $M_i$, $i<\delta$ be increasing continuous. $p_i\in S^{bs}(M_i)$ increasing and for $i<j<\delta$ we have $p_j\restriction_{M_i}=p_i$, all non-forking over $M_0$ and $p_\delta$ upper bound. Suppose $p_\delta$ has $\leq \lambda$-many realizations. Then there is a set $S$ of cardinality $\lambda^+$ of realizations of $p_0$, such that for each $a\in S$, by locality there is $i< \delta$ such that $a$ realizes $p_i$ but not $p_{i+1}$. By pigeonhole principle for some $i<\delta$ there are $\lambda^+$-many realizations of $p_i$ that are not realizations of $p_{i+1}$. Since there are $\leq \lambda$-many types in $S(M_{i+1})$ that have $\leq \lambda$-many realizations, there must be another type in $S(M_{i+1})$ with $\lambda^+$ realizations distinct from $p_{i+1}$, which contradicts $\lambda$-uniqueness of $p_{i+1}$. 

For the moreover part, if $p_0$ does not fork over $M_0$, so $p_0=p_\delta\restriction_{M_0}$ is $\lambda$-unique, i.e. $p_\delta$ does not fork over $M_0$.
\end{proof}

\begin{thm} [$2^{\lambda^+}>2^\lambda$] Let $K$ be an AEC categorical in $\lambda\geq LS(K)$ , and $1\leq I(\lambda^+,K)<2^{\lambda^+}$. $K_\lambda$ has NIP if and only if there is a w*-good $\lambda$-frame on $K$ except possibly without (Continuity). Moreover, 
\begin{enumerate}
    \item ($\ded \lambda=\lambda^+<2^\lambda$) If $\mathfrak s_{\lambda-unq}$ is $\lambda$-compact, then the w*-good frame satisfies in addition that if $p\in S^{bs}(M)$, then there is $N\geq_K M$ and $q\in S^{bs}(N)$ extending $p$ that does not fork over $N$. In particular, for any $N'\geq_K N$ there is $q'\in gS(N')$ extending $q$ that does not fork over $N$.
    \item if $K$ is $(<\lambda^+,\lambda)$-local, then $\mathfrak s_{\lambda-unq}$ has (Continuity).
\end{enumerate}

\end{thm}
\begin{proof}
The moreover part follows from Lemma \ref{malem}.
\end{proof}
\begin{question}\label{q2}
Are there other positive consequences of NIP that rely on weaker assumptions?
        
\end{question}

\section{Syntactic independence property}
In this section we assume tameness, and use Galois Morleyization to show that the negation of NIP leads to being able to encode subsets, as a parallel of first order independence property.
\begin{hypothesis}
Let $\kappa$ be an infinite cardinal and $K$ an AEC. Let $\tau=L(K)$ be its underlying language.
\end{hypothesis}
We first extend the definition of Galois types to longer lengths and set-valued domains.
\begin{mydef}
\begin{enumerate}
    \item $K^3:=\{(\bar a,A,N)\mid  N\in K, A\subseteq  |N|, \bar a\mbox{ is a sequence from }|N|\}$.

    \item For $(\bar a_0,A,N_0), (\bar a_1,A,N_1)\in K^3$, $(\bar a_0,A,N_0) E_{at}(\bar a_1,A,N_1)$ if there are $N\in K$, $f_0:N_0\to_A N$, and $f_1:N_1\to_A N$ $K$-embeddings such that $f_0(\bar a_0)=f_1(\bar a_1)$, $f_0\restriction_A=f_1\restriction_A$.
    \item $E$ is the transitive closure of $E_{at}$.\
    \item For $(\bar a,A,N)\in K^3$, \emph{the Galois type of $\bar a$ over $A$ in $N$} is $\type(\bar a/A,N):=[(\bar a,A,N)]_{E}$.
    \item For $N\in K$ and $A\subseteq |N|$, $\alpha$ an ordinal or $\infty$, $gS^{<\alpha}(A;N):=\{\type(\bar a/A,N)\mid (\bar a,A,N)\in K^3 \mbox{ and }\bar a\in {}^{<\alpha} |N|\}$. $gS^{\alpha}(A;N)$ is defined similarly.
    \end{enumerate}
\end{mydef}

\begin{remark}
  In the case where $A=|M|$ for $M\in K$, $\bigcup_{N\geq_K M} gS^1(|M|, N)$ is what we defined as $gS(M)$ in Definition \ref{galoistypes}.
\end{remark}
 The following technique first appeared in \cite{vmorleyization}, which allows one to work with Galois types in a syntactic way.

\begin{mydef}
Let $\kappa$ be an infinite cardinal and $K$ an AEC. The \emph{$(<\kappa)$-Galois Morleyization} of $K$ is $\hat K$, an AEC (except that the language might not be finitary) in a  ($<\kappa)$-ary language $\hat \tau$ extending $\tau$ such that:
\begin{enumerate}
    \item The structures and the substructure relation $\leq_{\hat K}$ in $\hat K$ are the same as $K$.
    \item For each $p\in gS^{<\kappa}(\emptyset)$, there is a predicate of the same length $R_p\in \hat \tau$. For each $M\in K$ and $\bar a\in |M|$, define $M\models R_p[\bar a]$ if and only if $\type(\bar a/\emptyset, M)=p.$ By extension, one can interpret quantifier-free $L_{\kappa,\kappa}(\hat \tau)$ formulas.
    \item The $(<\kappa)$-syntactic type of $\bar a\in^{<\kappa} |M|$ over $A\subseteq |M|$ is $\textbf{tp}_{\mbox{qf-} L_{\kappa,\kappa}(\hat \tau)}(\bar a/A,M)$, the set of all quantifier-free $L_{\kappa,\kappa}(\hat \tau)$ formulas with parameters from $A$ that $\bar a$ satisfies. For a particular quantifier-free $L_{\kappa,\kappa}(\hat \tau)$-formula $\phi(\bar x,\bar y)$,\\ $\textbf{tp}_\phi(\bar b/A,M):=\{\phi(\bar x,\bar a)\mid \bar a\in A, M\models \phi(\bar b,\bar a)\}$.
    \item For $M\in K$ and $A\subseteq |M|$,  $S^{<\alpha}_{\mbox{qf}-L_{\kappa,\kappa}(\hat \tau)}(A;M):=\{\textbf{tp}_{\mbox{qf-} L_{\kappa,\kappa}(\hat \tau)}(\bar b/A, M)\mid \bar b\in {}^{<\alpha}|M|\}$
\end{enumerate}
\end{mydef}

\begin{remark}
  There are $\leq 2^{<(LS(K)^++\kappa)}$ formulas in $\hat \tau$. 
\end{remark}

\begin{mydef}
    For a theory $T$ in first order logic, and $\Gamma$ a set of $T$-types, $\tau$ a language contained in the language of $T$, let $EC(T,\Gamma)$ denote the class of models of $T$ omitting all types in $\Gamma$. Let $PC(T,\Gamma,\tau)$ denote the class of models of $T$ omitting all types in $\Gamma$ as $\tau$-structures.
\end{mydef}

\begin{fact}
\cite[3.18(2)]{vmorleyization}
Under the notation of the previous definition, $K$ is $(<\kappa)$-tame if and only if for each ordinal $\alpha$, $M\in K$, $A\subseteq M$, $\type(\bar b/A,M)\mapsto \textbf{tp}_{\mbox{qf-} L_{\kappa,\kappa}(\hat \tau)}(\bar b/A,M)$ from $gS^\alpha(A;M)$ to $S^\alpha_{\mbox{qf}-L_{\kappa,\kappa}(\hat \tau)}(A;M)$ is bijective.
\end{fact}

\begin{notation}
    For any formula $\varphi$ and a condition $i$, $\varphi^i$ means $\varphi$ itself when $i$ holds, and $\neg \varphi$ otherwise. For example, at the end of the proof of the next theorem, the formula is $\phi(c_i,x)$ and the condition is $i\in w$. When $i\in w$ holds, $\phi(c_i,x)^{i\in w}$ is $\phi(c_i,x)$. When $i\notin w$, $\phi(c_i,x)^{i\in w}$ is $\neg \phi(c_i,x)$.
    
\end{notation}

\begin{thm}
Suppose $K$ is $(<\aleph_0)$-tame, $M\in K$, $C\subseteq |M|$, $\lambda:=|C|\geq \beth_3(LS(K))$ and $(\ded \lambda)^{2^{LS(K)}}= \ded \lambda$. Suppose $|gS^1(C;M)|>\ded \lambda$. Then there is $N\in K$, $\langle \bar a_n\in^m |N|\mid n<\omega \rangle$ and $\phi$ in the language of Galois Morleyization
such that for every $w\subseteq \omega$ there is $b_w\in |N|$ such that for all $i<\omega$, 
$$
N\models \phi(\bar a_i,b_w) \iff i\in w
$$
\end{thm}

\begin{proof}
Let $\hat K$ be the $(<\aleph_0)$ Galois Morleyization of $K$. Note that both classes have the same Galois types. By Shelah's Presentation Theorem $\hat K=PC(T,\Gamma, \hat \tau)$ with $|T|\leq 2^{LS(K)}$, with the language of $T$ containing $\hat \tau$. Then by tameness and the previous fact $|S^{1}_{\mbox{qf}-L_{\omega,\omega}(\hat \tau)}(C;M)|>\ded \lambda$, so for some quantifier-free formula $\phi(\bar y,x)$ in $ L_{\omega,\omega}(\hat \tau)$ with $|S_\phi(C;M)|> \ded \lambda$, since there are $\leq 2^{LS(K)}$-many quantifier-free $L_{\omega,\omega}(\hat \tau)$-formulas.

Without loss of generality $C=\lambda=|C|$. Let $\mu:=(\ded \lambda)^+$. For notational simplicity we view $S_{\phi}(C;
M)$ as $S$, a family of subsets of $ ^{\ell(\bar y)}C$, where 
$$
A\in S\iff \{\phi(\bar a, x)\mid \bar a\in A\}\in S_\phi(C;M).
$$
We also assume $\bar y$ has length $1$. The proof for other cases is similar.

\textbf{Claim: }For all $\alpha <\lambda$, if $|\{A\cap \alpha\mid A\in S\}|\geq\mu$, then $\alpha\geq (\beth_2(LS(K)))^+$.\\
\emph{Proof of Claim: }Suppose there is $\alpha< \lambda$, $|\{A\cap \alpha\mid A\in S\}|\geq \mu$. Since $\{A\cap \alpha\mid A\in S\}$ is the set of branches of the a subtree of $^{<\alpha}2$, $\ded \lambda< \mu\leq \ded |^{<\alpha} 2|\leq \ded 2^{|\alpha|}$, so $2^{|\alpha|}>\lambda\geq \beth_3(LS(K))$, so $|\alpha|> \beth_2(LS(K))$. Thus the claim holds.

We may assume $\lambda> \beth_2(LS(K))$ and for all $\alpha<\lambda$, $|\{A\cap \alpha\mid A\in S\}|<\mu$. If this holds, then we are done since $\lambda\geq \beth_3(LS(K))>\beth_2(LS(K))$. If not, replace $\lambda$ with smallest $\alpha<\lambda$ such that $|\{A\cap \alpha \mid A\in S\}|\geq \mu.$ By minimality for all $\beta<\alpha$, $|\{A\cap \beta \mid A\in S\}|<\mu$. Such $\alpha$ might be small, but by the claim $\alpha\geq (\beth_2(LS(K)))^+,$ and this is enough for the rest of the argument.

  For each $\alpha\leq \lambda$ let $S^0_\alpha:=\{\langle A\cap \alpha,\alpha\rangle \mid A\in S\}$. $\bigcup_{\alpha<\lambda}  S^0_{\alpha}$ is a tree when equipped with
$$
(A_1,\alpha_1)\leq (A_2,\alpha_2) \iff \alpha_1\leq \alpha_2 \land A_1=A_2\cap \alpha_1.
$$
Let
$$
S^1_\alpha:=\{s\in S^0_\alpha\mid |\{t\in S^0_\alpha\mid s\leq t\}|\geq \mu\},
$$
and 
$$
S^1_\lambda:=\{s\in S^0_\lambda\mid \forall \alpha<\lambda( s\restriction_\alpha\in S^1_\alpha)\}.
$$
We build \begin{enumerate}
    \item for $n<\omega$, $S_n\subseteq S^1_\lambda$, and
    \item for each $(B,i)$ such that $B\subseteq A$ for some $(A,\lambda)\in S_n$ and $i<\lambda$,
    \begin{enumerate}
        \item $\lambda>\alpha^{B}_i(n,0)>\ldots >\alpha^{B}_i(n,n-1)>i$, a sequence of ordinals,
        \item $(D^{(B,i)}_{u,n},\lambda)\in S^1_\lambda$ for each $u\subseteq n$, and
        
    \end{enumerate}
    \item $p_n\in S^{n+2^n}_T(\emptyset)$ for $n<\omega$
\end{enumerate} such that:

\begin{enumerate}
    \item $S_0=S^1_\lambda$;
    \item $|S_n|\geq \mu\geq (\beth_2(LS(K)))^+$ for all $n$;
    \item $S_{n+1}\subseteq S_n$ for all $n$;
    \item The variables of $p_n$ are $x_i$ for $i<n$ ordered naturally and $y_S$ for $S\subseteq n$;
    \item $p_n\subseteq p_{n+1}$ for all $n$. This means the $p_{n+1}$ restricted to $x_i$ for $i<n$ and $y_S$ for $S\subseteq n$ is equal to $p_n$;
    \item For all $n<m$, $(A,\lambda)\in S_n$ and $(B,\lambda)\in S_m$, $i,j\in \lambda$
    \begin{align*}
        p_n &=\textbf{tp}_{T} (\langle \alpha^{A\cap i}_i(n,0),\ldots \alpha^{A\cap i}_i(n,n-1)\rangle\concat \langle D^{(A\cap i,i)}_{w,n}\mid w\subseteq n\rangle/\emptyset, M)\\
        &=\textbf{tp}_{T} (\langle \alpha^{B\cap j}_j(m,0),\ldots \alpha^{B \cap j}_j(m,n-1)\rangle\concat \langle D^{(B\cap j,j)}_{w,m}\mid w\subseteq m\rangle/\emptyset, M);
    \end{align*}
    \item For all $(A,i)\in S_n$ and $w\subseteq n$,  $(A,i)\leq (D^{(A,i)}_{w,n},\lambda)$ and $\alpha^A_i(n,i)\in D^{(A,i)}_{w,n}\iff i\in w$. 
\end{enumerate}
\textbf{Construction: }We build these objects by induction on $n$. When $n=0$ let $D^{(\emptyset,0)}_{\emptyset,0}$ be any element in $S^1_\lambda$. Assume we have built $S_n$, $\alpha^{A\cap i}_i(n,j)$ for $(A,\lambda)\in S_n$ and $p_n$. 

Fix $s=(A,i)$ for $i<\lambda$ such that for some $B$, $A\subseteq B$ and $(B,\lambda)\in S_n$. Clearly $T_s:=\{t\in \bigcup_{\beta<\lambda} S^1_\beta\mid s\leq t$ and $t$ extends to an element in $S_n\}$ is a tree. For every $s\leq t\in S_n$, $B_t:=\{t^*\mid s\leq t^*\leq t\}$ is a branch of $T_s$, and $t_1\neq t_2\implies B_{t_1}\neq B_{t_2}$. Since $$|S^0_\lambda-S^1_\lambda|=|\bigcup_{\alpha<\lambda, s\in S^0_\alpha-S^1_\alpha}\{t\in S^0_\lambda\mid s\leq t\}|<\mu,$$ $T_s$ has $\geq \mu$-many branches, and hence $|T_s|>\lambda$. 
Then for some $i'$, $|T_s\cap S^1_{i'}|>\lambda$. 
Let $s_{ j}=(A_{j}, i')\in T_s\cap S^1_{i'}$ for $j<\lambda^+$. Since there are $\leq \lambda$ finite tuples of ordinals $<\lambda$, we may assume $\alpha_{i'}^{A_j}$ are the same for all $j$. Now let $\alpha^A_i(n+1,k):=\alpha^{A_j}_{i'}(n, k)$ for all $k<n$. Let $\alpha^A_i(n+1,n)$ be the least $\alpha$ such that $s_{0}(\alpha)\neq s_{1}(\alpha)$, i.e. $\alpha \in A_0-A_1$ or $\alpha\in A_1-A_0$. Without loss of generality assume the latter case. For $w\subseteq (n+1)$, let $D^{(A,i)}_{w,n+1}:=D^{(A_0,i')}_{w,n}$ if $n\notin w$ and $D^{(A,i)}_{w,n+1}:=D^{(A_1,i')}_{w,n}$ if $n\in w$.

Note that $i<\alpha^A_i(n+1,n)<i'<\alpha^A_i(n+1,n-1)<\ldots <\alpha^A_i(n+1,0)$. Since $|S_n|\geq (\beth_2(LS(K)))^+$, and there are $\leq \beth_2(LS(K))$ $T$-types, by the pigeonhole principle there is $S_{n+1}\subseteq S_n$, $|S_{n+1}|\geq (\beth_2(LS(K)))^+$ such that for all $(A,i)$, $ (B,j)\in S_{n+1}$,
$$\textbf{tp}_{T} (\langle \alpha^A_i(n,0),\ldots \alpha^A_i(n+1,n)\rangle\concat \langle D^{(A,i)}_{w,n+1}\mid w\subseteq n+1\rangle/\emptyset, M)$$ is the same,
and define this type to be $p_{n+1}$. This finishes the construction. Note that here since $D^{(A,i)}_{w,n+1}$ is an element of $S^1_\lambda\subseteq S^0_\lambda=S$, i.e. a $\phi$-type, the ``$T$-type'' of $D^{(A,i)}_{w,n+1}$ is just the $T$-type of a realization of it, which can be fixed at the beginning of the proof.
\begin{align*}
    T^*:=T\cup \{\phi(c_i,d_w)^{i\in w})\mid w\subseteq \omega\}\cup \{p_n(\langle c_i\mid i<n\rangle \concat \langle d_w\mid w\subseteq \omega\rangle )\mid n<\omega\}
\end{align*}
is consistent, and by Morley's method we are done.
\end{proof}

Similar to the order property, this analogue of the independence property for AECs also has a Hanf number $\beth_{(2^{LS(K)})^+}$.

\begin{thm}
    If $K$ can encode subsets of $\mu:=\beth_{(2^{LS(K)})^+}$, then it can encode subsets of any cardinal. That is, if there are $M\in K$, $\{a_i\mid i<\mu\}\subseteq |M|$, $\{b_w\mid w\subseteq \mu\}\subseteq |M|$ such that for all $w\subseteq \mu$, $$i\in w\iff \phi(a_i,b_w),$$ then we can replace $\mu$ above by any cardinal.
\end{thm}

\begin{proof}
    We fix $\hat K$ and $\phi$ as in the proof of the previous theorem. Let $\lambda=(2^{LS(K)})^+$. Suppose $K$ can encode subsets of $\mu:=\beth_{(2^{LS(K)})^+}$. That is, there are $M\in K$, $\{a_i\mid i<\mu\}\subseteq |M|$, $\{b_w\mid w\subseteq \mu\}\subseteq |M|$ such that for all $w\subseteq \mu$, $$i\in w\iff \phi(a_i,b_w).$$ For each $i_0<\ldots <i_{n-1}<\mu$ and $u\subseteq n$, choose some subset $w\subseteq \mu$ such that $i_j\in w \iff \phi(a_{i_j},b_w)\iff j\in u$, and let $b_{u,n}^{i_0,\ldots,i_{n-1}}:=b_w$.
    We build $\langle F_n\subseteq \mu\mid n<\omega\rangle$, $\langle X_{\xi,n}\subseteq\mu\mid \xi\in F_n, n<\omega\rangle$ and $p_n\in S^{n+2^n}_T(\emptyset)$ such that:
    \begin{enumerate}
        \item For all $n<\omega$, $|F_n|=\lambda$;
        \item $|X_{\xi,n}|>\beth_\beta (2^{LS(K)})$ when $\xi$ is the $\beta$-th element of $F_n$;
        \item $p_n(\langle a_{i_j}\mid j<n\rangle \concat \langle b^{i_0,\ldots, i_{n-1}}_{u,n}\mid u\subseteq n\rangle)$.
    \end{enumerate}
Let $F_0=\lambda$ and $X_{\xi,0}:=\mu$ for all $\xi$. Suppose we have constructed everything for stage $n$. Fix $g:\lambda\to F_n$ an increasing enumeration. Let $G_n:=\{g(\beta+n+1)\mid \beta<\lambda\}$. For each $\xi=g(\beta+n+1)\in G_n$, consider the map $\langle i_j\mid j<n\rangle \mapsto \textbf{tp}_T(\langle a_{i_j}\mid j<n+1\rangle \concat \langle b^{i_0,\ldots, i_{n}}_{u,n+1}\mid u\subseteq n+1\rangle/\emptyset, M)$ from $[X_{\xi,n}]^{n+1}$ (increasing $(n+1)$-tuples from $X_{\xi,n}$) to $S_T^{n+2^n}(\emptyset)$. Since $|X_{\xi,n}|>\beth_{\beta+n+1}((2^{LS(K)})^+)$, by the Erd\H os-Rado theorem, there is a monochromatic subset $X_{\xi, n+1}\subseteq X_{\xi,n}$ such that $|X_{\xi,n+1}|>\beth_\beta((2^{LS(K)})^+)$. I.e. there is a type $p_{\xi,n+1}$ such that for all $i_0<\ldots <i_n$, $\textbf{tp}_T(\langle a_{i_j}\mid j<n\rangle \concat \langle b^{i_0,\ldots, i_{n}}_{u,n+1}\mid u\subseteq n+1\rangle/\emptyset, M)=p_{\xi,n+1}$. By the pigeonhole principle there is $F_{n+1}\subseteq G_n$ of cardinality $\lambda$ and $p_{n+1}$ such that for all $\xi\in F_{n+1}$, $p_{\xi,n+1}=p_{n+1}$.

Then $$T^*:=T\cup \{\phi(c_i,d_w)^{i\in w})\mid w\subseteq \kappa\}\cup \{p_n(\langle c_{i_j}\mid j<n\rangle \concat \langle d_w\mid w\subseteq w\rangle )\mid n<\omega, i_0<\ldots<i_{n-1}<\kappa\}$$
is consistent for any cardinal $\kappa$. By Morley's method we are done.
\end{proof}

\begin{lem}[Morley's method]
 Let $T$ be a first order theory with built-in Skolem functions and $\Gamma$ a set of $T$-types. Let $\langle c_i\mid i<\alpha\rangle$ be new constants. Let $p_{S}$ be a $T$-type in $|S|$ variables for every finite subset $S$ of $\alpha$, and $T^*$ a theory not containing any of the new constants such that:
\begin{enumerate}
    \item $T^*\supseteq T\cup \{p_S(\langle c_{\gamma}\mid \gamma\in S \rangle ) \mid S\subseteq \alpha \text{ finite}\}$ is consistent;
    \item Each $p_S$ is realized in some $M\in EC(T,\Gamma)$.

\end{enumerate} 
    Then there is $N\in EC(T^*, \Gamma)$.
\end{lem}
\begin{proof}
Let $M$ be a model of $T^*$ and without loss of generality $M=EM(\{c_i\mid i<\alpha\})$. We show that $M$ omits all types from $\Gamma$. Suppose not, i.e. $a\in |M|$ realizes some $q\in \Gamma$. Write $a$ as $\tau^M(c_{i_0},\ldots, c_{i_k})$ for some term $\tau$ in the language of $T$. Let $S:=\{c^M_{i_0},\ldots , c^M_{i_k}\}$ and $\langle b_0,\ldots, b_k\rangle \subseteq N^*\in EC(T,\Gamma)$ realizing $p_S$. Then for some $\varphi(y)\in q$, $N^*\models \neg \varphi(\tau(b_0,\ldots,b_k))$. As $p_S$ is complete, $\neg \varphi(\tau(x_0,\ldots,x_k))\in p_S$. Thus $M\not\models  \varphi(\tau(c_{i_0},\ldots,c_{i_k}))$, i.e. $M\models \neg\varphi(a)$, so $a$ does not realize $q$.
\end{proof}

\bibliographystyle{amsalpha}
\bibliography{nip.bib}

\end{document}